\documentclass[reqno]{amsart}
\usepackage{amssymb}
\usepackage{amsmath}

\usepackage{amsfonts}

\newtheorem{theorem}{Theorem}
\theoremstyle{plain}

\newtheorem{lemma}{Lemma}

\newtheorem{remark}{Remark}

\numberwithin{equation}{section}

\begin{document}
\title[Global attractors]{Global attractors for the one dimensional wave
equation with displacement dependent damping }
\author{A.Kh.Khanmamedov}
\address{{\small Department of Mathematics,} {\small Faculty of Science,
Hacettepe University, Beytepe 06800}, {\small Ankara, Turkey}}
\email{azer@hacettepe.edu.tr}
\date{}
\subjclass[2000]{ 35B41, \ 35L05}
\keywords{Attractors, wave equations}

\begin{abstract}
We study the long-time behavior of solutions of the one dimensional wave
equation with nonlinear damping coefficient. We prove that if the damping
coefficient function is strictly positive near the origin then this equation
possesses a global attractor.
\end{abstract}

\maketitle

\section{INTRODUCTION}

In this paper, we consider the following Cauchy problem:%
\begin{equation}
u_{tt}+\sigma (u)u_{t}-u_{xx}+\lambda u+f(u)=g(x),\text{ \ }(t,x)\in
(0,\infty )\times R\text{,}  \tag{1.1}
\end{equation}%
\begin{equation}
u(0,x)=u_{0}(x),\text{ \ \ \ \ \ \ \ \ }u_{t}(0,x)=u_{1}(x),\text{ \ \ \ \ \
\ \ \ \ \ \ \ \ \ \ \ \ \ }x\in R\text{,}  \tag{1.2}
\end{equation}%
where $\lambda $ is a positive constant, $g\in L_{1}(R)+L_{2}(R)$ and
nonlinear functions $f(\cdot )$ and $\sigma (\cdot )$ satisfy the following
conditions:%
\begin{equation}
f\in C^{1}(R)\text{, \ \ \ }f(u)u\geq 0\text{, \ \ }\nvdash \text{ }u\in R%
\text{, \ \ \ \ \ \ \ \ \ \ \ \ \ }  \tag{1.3}
\end{equation}%
\begin{equation}
\sigma \in C(R)\text{, \ \ }\sigma (0)>0\text{,\ \ \ \ }\sigma (u)\geq 0%
\text{,\ \ }\nvdash \text{ }u\in R\text{.\ \ }  \tag{1.4}
\end{equation}

Applying standard Galerkin's method and using techniques of \cite[%
Proposition 2.2]{6}, it is easy to prove the following existence and
uniqueness theorem:

\begin{theorem}
Assume that the conditions (1.3)-(1.4) hold. Then for any $T>0$ and $%
(u_{0},u_{1})\in \mathcal{H}:=H^{1}(R)\times L_{2}(R)$ the problem
(1.1)-(1.2) has a\ unique weak solution $u\in $ $C([0,T];H^{1}(R))\cap
C^{1}([0,T];L_{2}(R))\cap C^{2}([0,T];H^{-1}(R))$ on $[0,T]\times R$ such
that%
\begin{equation*}
\left\Vert (u(t),u_{t}(t))\right\Vert _{\mathcal{H}}\leq c(\left\Vert
(u_{0},u_{1})\right\Vert _{\mathcal{H}})\text{, \ }\nvdash t\geq 0\text{,}
\end{equation*}%
where $c:R_{+}\rightarrow R_{+}$ is a nondecreasing function. Moreover if $%
v\in $ $C([0,T];H^{1}(R))\cap C^{1}([0,T];L_{2}(R))\cap
C^{2}([0,T];H^{-1}(R))$ is also weak solution to (1.1)-(1.2) with initial
data $(v_{0},v_{1})\in \mathcal{H}$, then
\begin{equation*}
\left\Vert u(t)-v(t)\right\Vert _{L_{2}(R)}+\left\Vert
u_{t}(t)-v_{t}(t)\right\Vert _{H^{-1}(R)}\leq
\end{equation*}%
\begin{equation*}
\leq \widetilde{c}(T,\widetilde{R})\left( \left\Vert u_{0}-v_{0}\right\Vert
_{L_{2}(R)}+\left\Vert u_{1}-v_{1}\right\Vert _{H^{-1}(R)}\right) \text{, \ }%
\nvdash t\in \lbrack 0,T]\text{,}
\end{equation*}%
where $\widetilde{c}:R_{+}\times R_{+}\rightarrow R_{+}$ is a nondecreasing
function with respect to each variable and $\widetilde{R}=\max \left\{
\left\Vert (u_{0},u_{1})\right\Vert _{\mathcal{H}},\left\Vert
(v_{0},v_{1})\right\Vert _{\mathcal{H}}\right\} $.
\end{theorem}

Thus, by the formula $(u(t),u_{t}(t))=S(t)(u_{0},u_{1})$, the problem
(1.1)-(1.2) generates a weak continuous semigroup $\left\{ S(t)\right\}
_{t\geq 0}$ in $\mathcal{H}$, where $u(t,x)$ is a weak solution of
(1.1)-(1.2), determined by Theorem 1.1, with initial data $(u_{0},u_{1})$.

The attractors for equation (1.1) in the finite interval were studied in
\cite{2}, assuming positivity of $\sigma (\cdot )$. For two dimensional
case, the attractors for the wave equation with displacement dependent
damping were investigated in \cite{7} under conditions%
\begin{equation*}
\sigma \in C^{1}(R)\text{, }0<\sigma _{0}\leq \sigma (u)\leq c(1+\left\vert
u\right\vert ^{q})\text{, \ }\nvdash u\in R\text{, }0\leq q<\infty \text{,}
\end{equation*}%
and
\begin{equation}
|\sigma ^{\prime }(u)|\leq c[\sigma (u)]^{1-\varepsilon }\text{, \ }\nvdash
u\in R\text{, \ \ }0<\varepsilon <1\text{,}  \tag{1.5}
\end{equation}%
on the damping coefficient. Recently, in \cite{3}, condition (1.5) has been
improved as
\begin{equation*}
|\sigma ^{\prime }(u)|\leq c\sigma (u)\text{, }\nvdash u\in R\text{.}
\end{equation*}%
For the three dimensional bounded domain case, the existence of a global
attractor for the wave equation with displacement dependent damping was
proved in \cite{6} when $\sigma (\cdot )$ is a strictly positive and
globally bounded function. In this case, when $\sigma (\cdot )$ is\ not
globally bounded, but is equal to a positive constant in a large enough
interval, the existence of a global attractor has been established in \cite%
{4}.

In the articles mentioned above, the existence of global attractors was
proved under positivity or strict positivity condition on the damping
coefficient function $\sigma (\cdot )$. In this paper, we study a global
attractor for (1.1)-(1.2) under weaker conditions on $\sigma (\cdot )$ and
prove the following theorem:

\begin{theorem}
Under conditions (1.3)-(1.4) a semigroup $\left\{ S(t)\right\} _{t\geq 0}$
generated by (1.1)-(1.2) possesses a global attractor in $\mathcal{H}$.
\end{theorem}

\section{Proof of Theorem 1.2}

To prove this theorem we need the following lemma:

\begin{lemma}
Let conditions (1.3)-(1.4) hold and let $B$ be a bounded subset of $\mathcal{%
H}$. Then for any $\varepsilon >0$ there exist $T_{0}=T_{0}(\varepsilon
,B)>0 $ and $r_{0}=r_{0}(\varepsilon ,B)>0$ such that
\begin{equation}
\left\Vert S(t)\varphi \right\Vert _{H^{1}(R\backslash (-r_{0},r_{0}))\times
L_{2}(R\backslash (-r_{0},r_{0}))}<\varepsilon ,\text{ }\nvdash t\geq T_{0},%
\text{ \ }\nvdash \varphi \in B.  \tag{2.1}
\end{equation}
\end{lemma}

\begin{proof}
Let $(u_{0},u_{1})\in B$ and $S(t)(u_{0},u_{1})=(u(t),u_{t}(t))$.
Multiplying (1.1) by $u_{t}$\bigskip\ and integrating over $(0,t)\times R$
we obtain%
\begin{equation}
\left\Vert u_{t}(t)\right\Vert _{L_{2}(R)}^{2}+\left\Vert u(t)\right\Vert
_{H^{1}(R)}^{2}+\underset{0}{\overset{t}{\int }}\underset{R}{\int }\sigma
(u(\tau ,x))u_{t}^{2}(\tau ,x)dxd\tau \leq c_{1},\text{ \ \ \ }\nvdash t\geq
0.  \tag{2.2}
\end{equation}%
Let $\eta \in C^{1}(R)$, $0\leq \eta (x)\leq 1$, $\eta (x)=\left\{
\begin{array}{c}
0,\text{ \ \ }\left\vert x\right\vert \leq 1 \\
1,\text{ \ }\left\vert x\right\vert \geq 2\text{\ }%
\end{array}%
\right. $, $\eta _{r}(x)=\eta (\frac{x}{r})$ and $\Sigma (u)=\underset{0}{%
\overset{u}{\int }}\sigma (s)ds$. Multiplying (1.1) by $\eta _{r}^{2}\Sigma
(u)$, integrating over $(0,t)\times R$ and taking into account (2.2) we have%
\begin{equation*}
\underset{0}{\overset{t}{\int }}\underset{R}{\int }\eta _{r}^{2}(x)\sigma
(u(\tau ,x))u_{x}^{2}(\tau ,x)dxd\tau +\lambda \underset{0}{\overset{t}{\int
}}\underset{R}{\int }\eta _{r}^{2}(x)\Sigma (u(\tau ,x))u(\tau ,x)dxd\tau
\leq
\end{equation*}%
\begin{equation}
\leq c_{2}(1+\sqrt{t}+\frac{t}{r}+t\left\Vert g\right\Vert
_{L_{1}(R\backslash (-r,r))+L_{2}(R\backslash (-r,r))}),\text{ \ \ \ }%
\nvdash t\geq 0,\text{ \ \ \ }\nvdash r>0.  \tag{2.3}
\end{equation}%
By (1.4), there exists $l>0$, such that
\begin{equation}
\frac{\sigma (0)}{2}\leq \sigma (s)\leq 2\sigma (0),\text{ \ \ \ \ }\nvdash
s\in \lbrack -l,l].  \tag{2.4}
\end{equation}%
Using embedding $H^{\frac{1}{2}+\varepsilon }(R)\subset L_{\infty }(R)$ and
taking into account (2.2) and (2.4) we find
\begin{equation*}
\underset{0}{\overset{t}{\int }}\underset{R}{\int }\eta
_{r}^{2}(x)u^{2}(\tau ,x)dxd\tau \leq \frac{2}{\sigma (0)}\underset{0}{%
\overset{t}{\int }}\underset{\{x:\left\vert u(\tau ,x)\right\vert \leq l\}}{%
\int }\eta _{r}^{2}(x)\Sigma (u(\tau ,x))u(\tau ,x)dxd\tau +
\end{equation*}%
\begin{equation*}
+c_{3}\underset{0}{\overset{t}{\int }}\underset{\{x:\left\vert u(\tau
,x)\right\vert >l\}}{\int }\eta _{r}^{2}(x)\left\vert u(\tau ,x)\right\vert
dxd\tau \leq \frac{2}{\sigma (0)}\underset{0}{\overset{t}{\int }}\underset{%
\{x:\left\vert u(\tau ,x)\right\vert \leq l\}}{\int }\eta _{r}^{2}(x)\Sigma
(u(\tau ,x))u(\tau ,x)dxd\tau +
\end{equation*}%
\begin{equation*}
+\frac{2c_{3}}{\sigma (0)l}\underset{0}{\overset{t}{\int }}\underset{%
\{x:\left\vert u(\tau ,x)\right\vert >l\}}{\int }\eta _{r}^{2}(x)\Sigma
(u(\tau ,x))u(\tau ,x)dxd\tau
\end{equation*}%
and consequently%
\begin{equation}
\underset{0}{\overset{t}{\int }}\left\Vert \eta _{r}u(\tau )\right\Vert
_{L_{\infty }(R)}^{5}d\tau \leq c_{4}\underset{0}{\overset{t}{\int }}%
\left\Vert \eta _{r}u(\tau )\right\Vert _{L_{2}(R)}^{2}d\tau \leq c_{5}%
\underset{0}{\overset{t}{\int }}\underset{R}{\int }\eta _{r}^{2}(x)\Sigma
(u(\tau ,x))u(\tau ,x)dxd\tau ,\text{ \ }  \tag{2.5}
\end{equation}%
for $r\geq 1$. So by (2.2), (2.3) and (2.5), we get
\begin{equation*}
\underset{0}{\overset{t}{\int }}\left[ \left\Vert \eta _{2r}\sigma ^{\frac{1%
}{2}}(u(\tau ))u_{t}(\tau )\right\Vert _{L_{2}(R)}^{2}+\left\Vert \eta
_{2r}\sigma ^{\frac{1}{2}}(u(\tau ))u_{x}(\tau )\right\Vert
_{L_{2}(R)}^{2}+\lambda \left\Vert \eta _{2r}\sigma ^{\frac{1}{2}}(u(\tau
))u(\tau )\right\Vert _{L_{2}(R)}^{2}\right. +
\end{equation*}%
\begin{equation}
+\left. \left\Vert \eta _{r}u(\tau )\right\Vert _{L_{\infty }(R)}^{5}\right]
d\tau \leq c_{6}(1+\sqrt{t}+\frac{t}{r}+t\left\Vert g\right\Vert
_{L_{1}(R\backslash (-r,r))+L_{2}(R\backslash (-r,r))}),\text{ \ }\nvdash
t\geq 0,\text{ }\nvdash r\geq 1.  \tag{2.6}
\end{equation}%
Now denote $\Phi _{r}(u(t)):=\frac{1}{2}\left\Vert \eta
_{r}u_{t}(t)\right\Vert _{L_{2}(R)}^{2}+\frac{1}{2}\left\Vert \eta
_{r}u_{x}(t)\right\Vert _{L_{2}(R)}^{2}+\mu \left\langle \eta _{r}u_{t}(t),%
\text{ }\eta _{r}u(t)\right\rangle +\frac{\lambda }{2}\left\Vert \eta
_{r}u(t)\right\Vert _{L_{2}(R)}^{2}+\left\langle \eta _{r}F(u(t)),\text{ }%
\eta _{r}\right\rangle +\left\langle \eta _{r}g,\text{ }\eta
_{r}u(t)\right\rangle $, where $\mu =\min \left\{ \sqrt{\frac{\lambda }{2}},%
\frac{\sigma (0)}{5},\frac{\lambda }{2\sigma (0)}\right\} ,$ $\left\langle u,%
\text{ }v\right\rangle =\underset{R}{\int }u(x)v(x)dx$ and $F(u)=\underset{0}%
{\overset{u}{\int }}f(s)ds.$\bigskip\ By (2.4) and (2.6), it follows that
for any $\delta >0$ there exist $T_{\delta }=T_{\delta }(B)>0$, $r_{1,\delta
}=r_{1,\delta }(B)>1$ and for any $r\geq r_{1,\delta }$ there exists $%
t_{\delta ,r}^{\ast }\in \lbrack 0,T_{\delta }]$ such that%
\begin{equation}
\Phi _{r}(u(t_{\delta ,r}^{\ast }))<\delta ,\text{ \ \ }\nvdash r\geq
r_{1,\delta }.  \tag{2.7}
\end{equation}%
Again by (2.2), we have%
\begin{equation*}
\left\Vert \eta _{r}u(t)\right\Vert _{L_{2}(R)}\leq \left\Vert \eta
_{r}u(t_{\delta ,r}^{\ast })\right\Vert _{L_{2}(R)}+\underset{t_{\delta
,r}^{\ast }}{\overset{t}{\int }}\left\Vert \eta _{r}u_{t}(s)\right\Vert
_{L_{2}(R)}ds\leq \left\Vert \eta _{r}u(t_{\delta ,r}^{\ast })\right\Vert
_{L_{2}(R)}+c_{7}(t-t_{\delta ,r}^{\ast })
\end{equation*}%
and consequently%
\begin{equation*}
\left\Vert \eta _{_{r}}u(t)\right\Vert _{L_{\infty }(R)}^{3}\leq
c_{8}\left\Vert \eta _{_{r}}u(t)\right\Vert _{L_{2}(R)}\leq c_{9}(\Phi _{r}^{%
\frac{1}{2}}(u(t_{\delta ,r}^{\ast }))+\left\Vert g\right\Vert
_{L_{1}(R\backslash (-r,r))+L_{2}(R\backslash (-r,r))}^{\frac{1}{2}%
}+t-t_{\delta ,r}^{\ast })<
\end{equation*}%
\begin{equation*}
<c_{9}(\delta ^{\frac{1}{2}}+\left\Vert g\right\Vert _{L_{1}(R\backslash
(-r,r))+L_{2}(R\backslash (-r,r))}^{\frac{1}{2}}+t-t_{\delta ,r}^{\ast }),%
\text{ }\nvdash t\geq t_{\delta ,r}^{\ast },\text{ }\nvdash r\geq
r_{1,\delta }.
\end{equation*}%
Denoting $T_{\delta ,r}^{\ast }=t_{\delta ,r}^{\ast }+\frac{l^{3}}{3c_{9}}$
and choosing $\delta \in (0,\frac{l^{6}}{9c_{9}^{2}})$, by the last
inequality, we can say that there exists $r_{2,\delta }\geq 2r_{1,\delta }$
such that%
\begin{equation}
\left\Vert u(t)\right\Vert _{L_{\infty }(R\backslash (-r_{2,\delta
},r_{2,\delta }))}<l,\text{ \ \ }\nvdash t\in \lbrack t_{\delta ,r}^{\ast
},T_{\delta ,r}^{\ast }].\text{ \ }  \tag{2.8}
\end{equation}%
Now multiplying (1.1) by $\eta _{r}^{2}(u_{t}+\mu u)$, integrating over $R$
and taking into account (2.4) and (2.8) we obtain%
\begin{equation*}
\frac{d}{dt}\Phi _{r}(u(t))+c_{10}\Phi _{r}(u(t))\leq c_{11}(\frac{1}{r}%
+\left\Vert g\right\Vert _{L_{1}(R\backslash (-r,r))+L_{2}(R\backslash
(-r,r))}),\text{ }\nvdash t\in \lbrack t_{\delta ,r}^{\ast },T_{\delta
,r}^{\ast }],
\end{equation*}%
and consequently
\begin{equation}
\Phi _{r}(u(t))\leq \Phi _{r}(u(t_{\delta ,r}^{\ast
}))e^{-c_{10}(t-t_{\delta ,r}^{\ast })}+c_{11}(\frac{1}{r}+\left\Vert
g\right\Vert _{L_{1}(R\backslash (-r,r))+L_{2}(R\backslash (-r,r))})\frac{%
1-e^{-c_{10}(t-t_{\delta ,r}^{\ast })}}{c_{10}},  \tag{2.9}
\end{equation}%
for $r\geq r_{2,\delta }$ . By (2.7) and (2.9), there exists $r_{3,\delta
}\geq r_{2,\delta }$ such that
\begin{equation*}
\Phi _{r}(u(t))<\delta ,\text{ \ \ }\nvdash r\geq r_{3,\delta },\text{ }%
\nvdash t\in \lbrack t_{\delta ,r}^{\ast },T_{\delta ,r}^{\ast }].
\end{equation*}%
Hence denoting by $n_{\delta }$ the smallest integer number which is not
less than $\frac{3c_{9}T_{\delta }}{l^{3}}$ and applying above procedure at
most $n_{\delta }$ time, we find
\begin{equation*}
\Phi _{r}(u(T_{\delta }))<\delta ,\text{ \ \ \ }\nvdash r\geq r_{4,\delta },
\end{equation*}%
for some $r_{4,\delta }\geq 2^{n_{\delta }}r_{1,\delta }$. From the last
inequality it follows that for any $\varepsilon >0$ there exist $\widehat{T}%
_{\varepsilon }=\widehat{T}_{\varepsilon }(B)>0$ and $\widehat{r}%
_{\varepsilon }=\widehat{r}_{\varepsilon }(B)>0$ such that
\begin{equation*}
\left\Vert S(\widehat{T}_{\varepsilon })\varphi \right\Vert
_{H^{1}(R\backslash (-\widehat{r}_{\varepsilon },\widehat{r}_{\varepsilon
}))\times L_{2}(R\backslash (-\widehat{r}_{\varepsilon },\widehat{r}%
_{\varepsilon }))}<\varepsilon ,\text{ \ }\nvdash \varphi \in B.
\end{equation*}%
Since, by (2.2), $B_{0}=\underset{t\geq 0}{\cup }S(t)B$ is a bounded subset
of $\mathcal{H}$, for any $\varepsilon >0$ there exist $T_{0}=T_{0}(%
\varepsilon ,B)>0$ and $r_{0}=r_{0}(\varepsilon ,B)>0$ such that%
\begin{equation*}
\left\Vert S(T_{0})\varphi \right\Vert _{H^{1}(R\backslash
(-r_{0},r_{0}))\times L_{2}(R\backslash (-r_{0},r_{0}))}<\varepsilon ,\text{
\ }\nvdash \varphi \in B_{0}.
\end{equation*}%
Taking into account positively invariance of $B_{0}$, from the last
inequality we obtain (2.1).
\end{proof}

By (2.1) and (2.4), for any bounded subset $B$ of $\mathcal{H}$ there exist $%
\widehat{T}_{0}=\widehat{T}_{0}(B)>0$ and $\widehat{r}_{0}=\widehat{r}%
_{0}(B)>0$ such that%
\begin{equation}
\sigma (u(t,x))\geq \frac{\sigma (0)}{2},\text{ \ \ \ \ }\nvdash t\geq
\widehat{T}_{0},\text{ \ }\nvdash \left\vert x\right\vert \geq \widehat{r}%
_{0}.  \tag{2.10}
\end{equation}%
Hence using techniques of \cite{5} one can prove the asymptotic compactness
of the semigroup $\left\{ S(t)\right\} _{t\geq 0}$, which is included in the
following lemma:

\begin{lemma}
Assume that conditions (1.3)-(1.4) hold and $B$ is bounded subset of $%
\mathcal{H}$. Then every sequence of the form $\left\{ S(t_{n})\varphi
_{n}\right\} _{n=1}^{\infty }$, $\left\{ \varphi _{n}\right\} _{n=1}^{\infty
}\subset B$, $t_{n}\rightarrow \infty $, has a convergent subsequence in $%
\mathcal{H}$.
\end{lemma}

By (2.10) and the unique continuation result of \cite{8}, it is easy to see
that problem (1.1)-(1.2) has a strict Lyapunov function (see \cite{1} for
definition). Thus according to \cite[Corollary 2.29]{1} the semigroup $%
\left\{ S(t)\right\} _{t\geq 0}$ possesses a global attractor.

\begin{remark}
We note that, for the problem considered in \cite{2}, from compact embedding
$H_{0}^{1}(0,\pi )\subset C[0,\pi ]$, it immediately follows that $\ \sigma
(u(t,x))\geq \frac{\sigma (0)}{2},$ \ \ $\nvdash t\geq 0,$ \ $\nvdash x\in
\lbrack 0,\varepsilon ]\cup \lbrack \pi -\varepsilon ,\pi ]$ \ for some $%
\varepsilon \in (0,\pi )$. So a global attractor still exists if one
replaces the positivity condition on $\sigma (\cdot )$ by the $\sigma (0)>0$.
\end{remark}

\end{document}